\documentclass{tran-l}

\newtheorem{theorem}{Theorem}[section]
\newtheorem{lemma}[theorem]{Lemma}
\newtheorem{proposition}[theorem]{Proposition}
\newtheorem{corollary}[theorem]{Corollary}
\theoremstyle{definition}
\newtheorem{definition}[theorem]{Definition}

\newtheorem{question}[theorem]{Question}
\theoremstyle{remark}

\numberwithin{equation}{section}

 \def\T{ \mathbb T}
 \def\R{ \mathbb R}
 \def\H{H^\infty}

 \def\D{{ \mathbb D}}
  
    \def\Q{{ \mathbb Q}}
 \def\C{{ \mathbb C}}
 \def\Z{{ \mathbb Z}}
 \def\N{{ \mathbb N}}

  \newcommand{\zit}[1]{(\ref{#1})}
  \def\ap {{\rm AP}^+}
 \def\bsr{\operatorname{bsr}}
\def\tsr{\operatorname{tsr}}

 \def\sp {\quad}

 \def\bs{\boldsymbol}
 \def\dis{\displaystyle}
 \def\union{\cup}
 
 \def\inter{\cap}
 
 \def\ov{\overline}

 \def\ss{\subseteq}
 \def\emp{\emptyset}

 \def\buildrel#1_#2^#3{\mathrel{\mathop{\kern 0pt#1}\limits_{#2}^{#3}}}

 \def\ssi{\Longleftrightarrow}

\begin{document}

 \title [Almost periodic functions]{The Bass and topological stable ranks for
 algebras  of almost periodic functions\\  on the real line}


 \author{Raymond Mortini}
  \address{
Universit\'{e} de Lorraine\\
 D\'{e}partement de Math\'{e}matiques et  
Institut \'Elie Cartan de Lorraine,  UMR 7502\\
 Ile du Saulcy\\
 F-57045 Metz, France} 
 \email{Raymond.Mortini@univ-lorraine.fr}

 \author{Rudolf Rupp}
\address{ Fakult\"at f\"ur Angewandte Mathematik, Physik  und Allgemeinwissenschaften\\
\small TH-N\"urnberg\\
\small Kesslerplatz 12\\
\small D-90489 N\"urnberg, Germany
}
\email  {Rudolf.Rupp@th-nuernberg.de}

\thanks{We thank Amol Sasane for his  interest in this work and  
 Albrecht B\"ottcher for  the reference  \cite{bks}. We also thank the referee for his suggestions.
}


\begin{abstract}
Let $\Lambda$  be a  sub-semigroup of the reals. 
We show that  the Bass and topological stable ranks of the algebras  
${\rm AP}_\Lambda=\{f\in {\rm AP}: \sigma(f)\ss \Lambda\}$
of almost periodic functions  on the real line and with Bohr spectrum in $\Lambda$
are infinite whenever the algebraic dimension of the $\Q$-vector space 
 generated by $\Lambda$ is infinite.
This extends Su\'arez's  result for ${\rm AP}_\R={\rm AP}$. Also considered are 
general subalgebras of AP. 
\end{abstract}

\maketitle

\section*{Introduction}
 
 Let $C_b(\R,\C)$ denote the set of bounded, continuous functions on $\R$ with values in $\C$
 and  let AP  be the uniform closure in $C_b(\R,\C)$ of  the set of all functions of the form
 $$Q(t):=\sum_{j=1}^N a_j e^{i\lambda_j t},$$
 where $a_j\in \C$, $\lambda_j\in \R$ and $N\in \N^*$. AP is  
 the set of {\it almost periodic functions}.
 We call $Q$ a {\it generalized trigonometric polynomial}.
 Under the usual pointwise algebraic operations, AP is a point separating function algebra
 on $\R$ with the property that $f\in {\rm AP}$ implies that $\ov f \in {\rm AP}$.
 Harald Bohr developed the  basic theory  for this space in a series of papers \cite{boh,boh2}.
 We also refer to the nice books by Corduneanu \cite{cor}  and Besicovich \cite{bes} for an introduction  into this important class of functions.  Modern treatments and applications
 to operator theory can be found for example in \cite{bks}.
 
 In our paper we are interested in a specific algebraic property and its topological counterpart
 for subalgebras of AP:
 namely the Bass and topological stable ranks (see below for the definition).
 In \cite{sua} Daniel Su\'arez showed that these ranks are infinite  for AP.
 This result reflected again the close connection  of AP to the infinite polydisk algebra $A(\D^\infty)$,
 a connection unveiled by Bohr in his  fundamental papers.  Let us mention that
 the stable ranks of $A(\D^\infty)$ were known to be infinite  (see for instance
 \cite{mo}).  Su\'arez's  approach did not allow to calculate the Bass stable rank of 
  the analytic trace of AP,  namely the  algebra 
 $$\mbox{$\ap=\{f\in {\rm AP}:  \widehat f(\lambda)=0 $ for all $\lambda\in \;]-\infty,0[\;\}$}$$
 of those almost periodic functions on $\R$ \footnote{Here $\widehat f(\lambda)$ is the  Fourier-Bohr coefficient associated with $\lambda$;  see below for the Definition.} 
 that admit a bounded holomorphic extension to
 the upper half-plane.   Indeed, he deduced the fact that $\bsr {\rm AP}=\infty$ from 
his Theorem that the topological stable rank of {\rm AP} is infinite  and Rieffel's result 
 that for commutative $C^*$-algebras
 the Bass stable rank coincides with the topological one (see \cite{ri}).
  Since in $\ap$ the only functions
$f$ satisfying  $f\in \ap$ and $\ov f\in \ap$ are the constants, Su\'arez's method  cannot be used 
and adapted to handle the algebra $\ap$.

 Using a different and actually more elementary  method than Su\'arez's,  we extend
 his result to the following natural subalgebras of AP: for  a sub-semigroup $\Lambda$ 
 of $(\R,+)$ let
 $${\rm AP}_\Lambda=\{f\in {\rm AP}: \sigma(f)\ss \Lambda\},$$
 where $\sigma(f)$ denotes the Bohr spectrum of $f\in{\rm AP}$ (see below).
 Our main result  tells us that  the Bass and topological stable ranks of ${\rm AP}_\Lambda$
 are infinite as well, whenever the dimension of the vector space $[\Lambda]$ 
 generated by $\Lambda$ over $\Q$ is infinite. Moreover, if this latter condition is not satisfied,
 then these stable ranks can be finite.    We also consider more general subalgebras of AP.
 
 Finally, we would like to  mention that a first attempt to calculate the stable ranks for
  ${\rm AP}_\Lambda$
 and, in particular, the one for $\ap$, was given by Mikkola and Sasane in \cite{misa}.
 
 \section{The tools for our proof}

  For the reader's convenience, we state here several results, surely known to people
 working with AP functions, that are necessary to understand  our proofs of the new results. 
 The most basic tool will be Kronecker's  approximation theorem. A very elegant proof can be found
 in \cite{klk}. 

  \begin{theorem}[Kronecker]\label{kron2}
  The following  statements are true: \footnote{ Here [C] stands for continuous, and  [D]  for discrete.}
  \begin{enumerate}
  
\item [(C$_N$)]  For $j=1,\dots, N$, let $\lambda_j\in \R$. Suppose that 
$\{\lambda_1,\dots,\lambda_N\}$ is
  linearly independent over $\Q$. Then 
  $$C:=\{\bigl(e^{i \lambda_1 t}, \dots, e^{i\lambda_N t}\bigr): t\in \R\}$$
  is dense in $\T^N$.
  
   \item [(D$_N$)]  For $j=1,\dots, N$, let $\lambda_j\in \R$. Suppose that 
   $\{\lambda_1,\dots,\lambda_N, 2\pi\}$ is
  linearly independent over $\Q$. Then 
  $$D:=\{\bigl(e^{i \lambda_1 n}, \dots, e^{i\lambda_N n}\bigr): n\in \N\}$$
  is dense in $\T^N$.
\end{enumerate}
\end{theorem}

Let us recall the definitions of the fundamental  notions in connection with almost
periodic functions (see for instance \cite{cor}).
 
   \begin{definition}\label{fbco}
   Let $f\in{\rm AP}$. If $\lambda\in\R$, the associated {\it Fourier-Bohr coefficient}
   $\widehat f(\lambda)$ is defined as
   $$ \widehat f(\lambda)=\lim_{|I|\to\infty} \frac{1}{|I|}\int_I f(t)e^{-i\lambda t}\, dt,$$
   where $I$ runs through  the set of all compact intervals in $\R$.
   \end{definition}
   
   \begin{proposition}\label{fbcs}
   If $f\in{\rm AP}$, then $\widehat f(\lambda)$ exists for every $\lambda\in \R$ and
   $\widehat f(\lambda)\not=0$ for at most a countable number of $\lambda$.
\end{proposition}

   \begin{definition}
   If  $f\in{\rm AP}$, then the {\it Bohr-spectrum}, $\sigma(f)$, of $f$ is the set  of all $\lambda\in\R$
  for which the  associated   Fourier-Bohr coefficient $\widehat f(\lambda)$ is not zero. If $\sigma(f)=\{\lambda_n: n\in I\}$, $I\ss\N$,  then the
    {\it Fourier-Bohr series} associated with $f$ is the formal series
    $$f\sim \sum_{n\in I} \widehat f(\lambda_n) e^{i\lambda_n t}.$$
\end{definition}

The main tool will be the following approximation theorem (see \cite{cor} for a classroom proof and
 \cite{bks} for a shorter, but more advanced   proof).

  \begin{theorem}\label{mainap} \hfill
 \begin{enumerate}
  \item [(1)] The Fourier-Bohr series uniquely determines $f$ whenever $f\in{\rm AP}$.
  \item [(2)] Let $f$ be an almost periodic function on $\R$ with Bohr spectrum $\sigma(f)$.
 Then  there exists a sequence $(q_n)$ of generalized trigonometric polynomials with
 $\sigma(q_n)\ss \sigma(f)$ converging uniformly to $f$.
 \end{enumerate}
\end{theorem}


\section{Some algebraic properties of  ${\rm AP}$ and ${\rm AP}_\Lambda$} 

We first present a simple proof of the well-known ``corona theorem''  for the algebra
of almost periodic functions.   

 \begin{definition}
Let $A$ be  a commutative unital algebra (real or complex)  (or just a commutative unital ring)
 with identity element denoted by 1.
 An $n$-tuple $(f_1,\dots,f_n)\in A^n$ is said to be {\it invertible} (or {\it unimodular}), 
 if there exists
 $(x_1,\dots,x_n)\in A^n$ such that the B\'ezout equation $\sum_{j=1}^n x_jf_j=1$
 is satisfied.
   The set of all invertible $n$-tuples is denoted by $U_n(A)$. Note that $U_1(A)=A^{-1}$,
   the latter being the set of invertible elements in $A$.
\end{definition}

\begin{theorem}\label{unap}  The following assertions hold:

\begin{enumerate}

\item [i)]  An element $F\in {\rm AP}$ is invertible if and only if  $\delta:=\inf_\R |F|>0$.

\item [ii)]  $U_n({\rm AP})=\{(F_1,\dots,F_n)\in {\rm AP}^n: \inf_\R \sum_{j=1}^n |F_j|>0\}.$

\end{enumerate}
\end{theorem}
\begin{proof}
i)  Let $F\in U_1({\rm AP})$. Then $F$ is invertible in $C_b(\R,\C)$ and so 
 $F$ is bounded away from zero. Conversely, let $F\in {\rm AP}$ satisfy $|F|\geq \delta>0$ on $\R$.
Since $F$  is bounded, there exists $M>0$ such that $\delta\leq |F|\leq M$ on $\R$.
 By Weierstrass' approximation theorem, let $(p_n)$ be  a sequence of polynomials
 in two variables such that
 $$\sup_{\delta\leq |z|\leq M} \left|p_n(z,\ov z)- \frac{1}{z}\right|\to 0.$$
 Then
 $$\sup_{x\in \R} \left| p_n\Bigl(F(x),\ov {F(x)}\Bigr)-\frac{1}{F(x)}\right|\to 0.$$
 Since $F\in {\rm AP}$ implies $\ov F\in {\rm AP}$, we deduce that $p_n(F,\ov F)$
 belongs to the algebra ${\rm AP}$, too. Because AP is uniformly closed we obtain
 $1/F\in {\rm AP}$.
 \\ 
  
 ii)  Let $\bs F:=(F_1,\dots, F_n)\in U_n({\rm AP)}$.  Then $\bs F$ is  an invertible $n$-tuple
 in $C_b(\R,\C)$. Hence, there is $\delta>0$ such that
 $ \sum_{j=1}^n |F_j|\geq \delta>0$. Conversely, if this latter condition is satisfied,
 then the function
 $$Q_j:=\frac{\ov F_j}{\sum_{k=1}^n |F_k|^2}$$
 is in $C_b(\R,\C)$. Since $F\in {\rm AP}$ implies $\ov F\in {\rm AP}$,  we see that 
 $|F|^2\in{\rm AP}$. Hence   $\sum_{k=1}^n|F_k|^2\in {\rm AP}$.
 By i) its inverse belongs to {\rm AP}, as well. Thus $Q_j\in {\rm AP}$.
 Since $\sum_{j=1}^n Q_jF_j=1$, we have shown that $\bs F\in U_n({\rm AP})$.
 \end{proof}

 \begin{lemma}\label{cnap+}
 Let $\Lambda$ be a sub-semigroup  of $(\R,+)$; that is $\Lambda$ has the following properties:
\begin{enumerate}
\item [(i)]  $0 \in \Lambda$;
\item [(ii)] $\lambda,\mu\in \Lambda$ implies $\lambda+\mu\in \Lambda$;
\end{enumerate}
Furthermore, let
$${\rm AP}_\Lambda:=\{f\in {\rm AP}: \sigma(f)\ss \Lambda\}.$$
Then 
\begin{enumerate}
 \item [(1)] ${\rm AP}_\Lambda$ is a uniformly closed subalgebra of {\rm AP}.
 \item [(2)]  If  $\emp\not=\Lambda_0=\{\lambda_1,\dots,\lambda_N\}\ss \Lambda\inter\R^+$, then  
the evaluation map 
$$\Phi_{\Lambda_0}: \begin{cases} A(\D^N) &\to {\rm AP}_\Lambda\\
                                         \hspace{0,5cm}f&\mapsto \Phi_{\Lambda_0}(f),
                                         \end{cases}
                                         $$
where $\Phi_{\Lambda_0} (f)(t) := f(e^{i\lambda_1t}, \dots,    e^{i\lambda_Nt})$ is an  algebra
 homomorphism         and $||\Phi_{\Lambda_0}(f)||_\infty\leq ||f||_\infty$.
 \item [(3)] If the positive  numbers 
 $\lambda_1,\dots, \lambda_N$      are linearly independent over $\Q$,
 then $\Phi_{\Lambda_0}$ is injective and $||\Phi_{\Lambda_0}(f)||_\infty   =||f||_\infty.$                   
\end{enumerate}
\end{lemma}

\begin{proof}
(1) We first show that ${\rm AP}_\Lambda$ is uniformly closed. In fact, if $(f_n)$ is a sequence
in ${\rm AP}_\Lambda$ coverging uniformly to some $f\in C_b(\R,\C)$, then $f\in {\rm AP}$
(because ${\rm AP}_\Lambda\ss {\rm AP}$ and AP is closed). Hence, by a standard reasoning, $\widehat f_n(\lambda)\to \widehat f(\lambda)$ for every $\lambda\in \R$. Consequently, 
 $\widehat f(\lambda)=0 $ for every  $\lambda\in \R\setminus \Lambda$; that is 
 $f\in {\rm AP}_\Lambda$. 
 
 Now we show that ${\rm AP}_\Lambda$ is an algebra.
 For $j=1,2$,  let $f_j\in {\rm AP}_\Lambda$. By Theorem \ref{mainap}(2), there is a sequence
 of trigonometric polynomials $p^{(j)}_n$ with $\sigma(p^{(j)}_n)\ss \sigma(f_j)\ss \Lambda$ 
 converging uniformly to $f_j$.  Now 
 $$\sigma(p^{(1)}_n+p^{(2)}_n)\ss \sigma(p^{(1)}_n) \union \sigma(p^{(2)}_n)\ss \Lambda.$$
 Since $\bigl(p^{(1)}_n +p^{(2)}_n\bigr)$ converges uniformly to $f_1+f_2$,
 we obtain  that $\sigma(f_1+f_2)\ss \Lambda$. Hence 
 $f_1+f_2\in {\rm AP}_\Lambda$. 
 
 Moreover,
 $$\sigma(p^{(1)}_n \cdot p^{(2)}_n)\ss [[\sigma(p^{(1)}) \union \sigma(p^{(1)})]]\ss \Lambda,$$
 where  $[[X]]$ denotes the  set $\{a+b:a,b\in X\}$.
 Hence, by a similar reasoning as above,
 $f_1\cdot f_2\in {\rm AP}_\Lambda$.
 
 Since $\alpha f\in {\rm AP}_\Lambda$ whenever $f\in {\rm AP}_\Lambda$ and $\alpha\in \C$,
 we conclude that ${\rm AP}_\Lambda$ is an algebra over $\C$.
 
 (2) This is  a consequence of the fact that every $f\in  A(\D^N)$ is the uniform limit of 
 a sequence of polynomials in $\C[z_1,\dots,z_N]$.
 
 (3)  
Using the hypothesis that $\{\lambda_1,\dots,\lambda_N\}$ is linearly independent over $\Q$,
 we obtain from Kronecker's
approximation Theorem \ref{kron2} that 
$$E:=\{(e^{i\lambda_1t},\dots, e^{i\lambda_Nt}): t\in \R\}$$
 is dense in $\T^N$.
Now $\sup_S |f|=\sup_{\T^N}|f|$ for every dense set $S$ in $\T^N$. Thus, for $f\in  A(\D^N)$,
and $\bs w=(w_1,\dots, w_N)$,
$$\sup_{t\in \R}|\Phi_\Lambda (f)(t)|=\sup_{t\in \R} |f(e^{i\lambda_1 t},\dots,  e^{i\lambda_N t})|
=\sup_{\bs w\in E} |f(w_1,\dots,w_N)|=\max_{\T^N}|f|.$$
By the distinguished maximum principle, we have   for every $f\in A(\D^N)$  that
 $$\max_{\T^N}|f|= ||f||_\infty.$$
Hence $||\Phi_\Lambda(f)||_\infty=||f||_\infty$. The injectivity of the linear map $\Phi_\Lambda$ follows.
 \end{proof}

 Here is the analogue  of Theorem \ref{unap} for $AP_{\Lambda}$ whenever
 $\Lambda$ is a subgroup of $(\R,+)$. \footnote{ The result itself is not new and we present it only
 since we could not  pin down our proof in the literature.}

 \begin{proposition} 
 Let $\Lambda\ss\R$ be an additive group. Then ${\rm AP}_\Lambda$  is a $C^*$-subalgebra of
 $C_b(\R,\C)$ and the following assertions hold:

\begin{enumerate}

\item [i)]  An element $F\in {\rm AP}_\Lambda$ is invertible if and only if  $\delta:=\inf_\R |F|>0$.
\item [ii)]  $U_n({\rm AP_\Lambda})=\{(F_1,\dots,F_n)\in {({\rm AP}_\Lambda)}^n: \inf_\R \sum_{j=1}^n |F_j|>0\}.$
\end{enumerate}
\end{proposition}

\begin{proof}
Since $\lambda\in \sigma (\ov f)\ssi -\lambda\in \sigma (f)$, 
the assumption ``$\Lambda$ a group''  implies that
 $\ov f\in {\rm AP}_\Lambda$  whenever $f\in {\rm AP}_\Lambda$.
 Hence ${\rm AP}_\Lambda$ is a $C^*$- subalgebra of $C_b(\R,\C)$.
 The remaining assertions now follow verbatim as in Theorem \ref{unap}.
\end{proof}

 If $\Lambda$ merely is a sub-semigroup, the situation is much more difficult. 
 For corona theorems in this setting, see \cite{boe} and  \cite{bks}.\bigskip

 For $\Lambda\ss\R$, let $[\Lambda]$ denote the $\Q$-vector space generated by  $\Lambda$.

 \begin{lemma}\label{trafo}
Let $Q(t)=\sum_{j=1}^N a_j e^{i\lambda_j t}$ be a generalized trigonometric polynomial.
Then for every  $\Q$-linearly independent subset $\Omega:=\{\omega_1,\dots, \omega_M\}$ of
 $\Lambda=\{\lambda_1,\dots, \lambda_N\}$  with the property that  $[\Omega]=[\Lambda]$, there is  $s\in \N^*$, independent of the coefficients $a_j$, and $q\in C(\T^M,\C)$ such that 
 $$\Phi_{\frac{\Omega}{{}^s}}(q)=Q,$$
  that is
$$ q\bigl(e^{i (\omega_1/s)t}, \dots, e^{i (\omega_M/s)t}\bigr)= Q(t).
$$

\end{lemma}
\begin{proof}
If $\Lambda$ itself is $\Q$-linearly independent, then the uniquely determined function
$$q(w_1,\dots, w_N)=\sum_{j=1}^N a_j w_j$$
satisfies $\Phi_\Lambda(q)=Q$ (see Lemma \ref{cnap+}).
Modulo a re-enumeration, let $\Omega=\{\lambda_1,\dots,\lambda_M\}$ and let $M<j\leq N$. 
Since $\lambda_j\in [\Omega]$, there are $s_j\in \N^*$ and  $s_{n,j}\in \Z$ such that
$$s_j\lambda_j=\sum_{n=1}^M s_{n,j} \lambda_n.$$
Let $s=\prod_{j=M+1}^N s_j$ and 
$ \dis r_{n,j}:=s_{n,j}\,\prod_{k=M+1\atop k\not=j}^N s_k.$
 Then
$$s \lambda_j=\sum _{n=1}^M r_{n,j} \lambda_n,~~ (j=M+1,\dots, N).$$
Hence, whenever $z_j=e^{i  (\lambda_j/s)  t}$, $(j=1,\dots, n)$, 
\begin{eqnarray*}
Q(t)&=&\sum_{j=1}^M a_j e^{i\lambda_jt} +\sum_{j=M+1}^N a_j\, 
e^{i \sum_{n=1}^M (r_{n,j}/s) \lambda_n t}\\
&=&\sum_{j=1}^M a_j e^{i\lambda_jt} + \sum_{j=M+1}^N a_j\, \prod_{n=1}^M 
e^{i (\lambda_n/s) r_{n,j} t}\\
&=& \sum_{j=1}^M a_j z_j^s + \sum_{j=M+1}^N a_j \prod_{n=1}^M  z_n^{r_{n,j}}\\
&=:&  q(z_1,\dots,z_M).
\end{eqnarray*}
We deduce that $\Phi_{\Omega'}(q)=Q$, whith $\Omega'=\{\omega_j/s: j=1,\dots,M\}$.
\end{proof}

 \section{The stable ranks of ${\rm AP}_\Lambda$}
 
 \begin{definition}
 
Let $A$ be  a commutative unital algebra (real or complex)  (or just a commutative unital ring)
 with identity element denoted by 1.
 \begin{enumerate}
 \item [i)]
An $(n+1)$-tuple $(f_1,\dots,f_n,g)\in U_{n+1}(A)$ is  called {\sl reducible}  
 \index{reducible tuple}
 if there exists 
 $(a_1,\dots,a_n)\in A^n$ such that $(f_1+a_1g,\dots, f_n+a_ng)\in U_n(A)$.
\item [(2)] The {\sl Bass stable rank} of $A$, denoted by $\bsr A$,  is the smallest integer $n$ such that every element in $U_{n+1}(A)$ is reducible. 
 If no such $n$ exists, then $\bsr A=\infty$. \index{$\bsr A$}
  \end{enumerate}
  \end{definition}
 
  \begin{definition}
Let $A$ be  a commutative unital Banach algebra.
 The {\it topological stable rank}, $\tsr A$, of $A$ is the least integer
  $n$ for which $U_n(A)$ is dense in $A^n$, or infinite if no such $n$ exists.  
  \end{definition}

It is well known that for Banach algebras, $\bsr A\leq \tsr A$ and that strict inequality
is possible (mostly for Banach algebras of holomorphic functions such as the disk algebra or $\H(\D)$.)
In order to determine the stable ranks of  ${\rm AP}_\Lambda$ we move to the polydisk algebra
and apply  Lemma \ref{cnap+}.

\begin{lemma}\label{decotorus2}\hfill
\begin{enumerate}
 \item[(i)]  
 Let $s\in \N^*$ and let $z^{1/s}=\exp(\frac{1}{s}\log z)$ be the canonical
  branch of the $s$-th root of $z$ on $\C\setminus \;]-\infty, 0]$. 
Let $r=|z|$ and $\theta=\arg z$ where $-\pi<\theta<\pi$. Then
$$g_s:\begin{cases}  2\ov \D\setminus [-2,0]&\to  \T^2\\ z&\mapsto
  \bigl( e^{i \frac{\arccos (r/2)}{^s}} e^{i\theta/s}, e^{-i\frac{\arccos (r/2)}{^s}} e^{i\theta/s}\bigr)
  \end{cases}$$
  is a continuous map, where $\arccos:[-1,1]\to [0,\pi]$ is the  standard inverse of the cosine function.
 \item [(ii)]  The map $$f_s: \begin{cases}\T^2 &\to 2\ov \D\\ 
 (z_1,z_2)&\mapsto   z_1^s+z_2^s\end{cases}$$
 is a continuous surjection  such that $f_s\circ g_s={\rm id}$ on $2\ov \D\setminus [-2,0]$.

\end{enumerate}
\end{lemma}
The straightforward proof is left to the reader.

Our key to the calculation of the Bass stable rank of ${\rm AP}_\Lambda$  
will be the following class of examples of non-reducible tuples in the polydisk algebra.
Let us emphasize that  non-reducibility in $A(\D^{N})$ will not be sufficient;
we need invertible tuples in $A(\D^{N})$ that are non-reducible in 
$C(\T^{N},\C)$,  which is a stronger property. In fact $(z_1, 1-z_1z_2)$ is an invertible tuple
that is  non-reducible in $A(\D^2)$ (since $F(z_1, \ov z_1):=z_1+h(z_1,\ov z_1) (1-|z_1|^2)$ is an extension
of the identity map on $\T$, and so has a zero in $\D$), but of course, 
$F(z_1,z_2)=z_1+h^*(z_1,z_2)(1-z_1z_2)\not=0$ on $\T\times \T$ whenever 
$h^*\equiv 0$.

 \begin{lemma}\label{fundamentalex}
For $j=1,\dots,2N$ and $s\in \N^*$,  let
$$f_j(z_1,\dots, z_{4N})=z^s_{2j-1}+z^s_{2j}-1,$$
and 
$$g=\frac{1}{4}-\sum_{j=1}^N f_j f_{N+j}.$$
Then $\bs F:=(f_1,\dots,f_N, g)$ is an invertible $(N+1)$-tuple in $A(\D^{4N})$ that is neither
reducible in $A(\D^{4N})$ nor in $C(\T^{4N},\C)$.
\end{lemma}

\begin{proof}
It is clear that $\bs F\in U_{N+1}(A(\D^{4N}))$. Let $\bs h=(h_1,\dots, h_N)\in C(\T^{4N},\C)^N$ and consider in $C(\T^{4N},\C)$  the $N$-tuple
$$ \bs H:=\Bigl(f_1+ h_1g, \dots,  f_N+h_Ng\Bigr).
$$
We claim that $\bs H(z_1^{(0)},\dots, z_{4N}^{(0)})=\bs 0_{N}$ for some
$(z_1^{(0)},\dots, z_{4N}^{(0)})\in \T^{4N}$. To this end,  we make the following  transformations.
 
  Let $z_{j+2N}=\ov z_j$ for $j=1,\dots, 2N$. Then, with
$\bs \xi:=(z_1,\dots, z_{2N}, \ov z_{1},\dots, \ov z_{2N})$,
$$\bs H(\bs\xi)=
\biggl( f_1(\bs\xi)+h_1(\bs\xi)\; \Bigl(\frac{1}{4}- \sum_{j=1}^N |f_j(\bs\xi)|^2\Bigr),\; \dots,\;
f_N(\bs\xi)+h_N(\bs\xi)\; \Bigl(\frac{1}{4}- \sum_{j=1}^N |f_j(\bs\xi)|^2\Bigr)
\biggr).
$$
Let  $D:=\{\zeta\in\C: |\zeta-1|\leq 1/2\}$. Then $D\ss 2\ov\D \setminus [-2,0]$.
If $u_{2j-1}\in D$, $j=1,\dots, N$,  then, using Lemma \ref{decotorus2} and $g_s=(G_1,G_2)$,
we put $z_{2j-1}:=G_1(u_{2j-1})$,   $z_{2j}:=G_2(u_{2j-1})$, and
$$h_j^*(u_1,u_3,\dots, u_{2N-1}):=$$
$$ h_j\Bigl(G_1(u_1), G_2(u_1),  \dots, G_1(u_{2N-1}), G_2(u_{2N-1}),\;
\ov G_1(u_1), \ov G_2(u_1),  \dots, \ov G_1(u_{2N-1}), \ov G_2(u_{2N-1})
\Bigr).$$
Since $u_{2j-1}=z_{2j-1}^s+z_{2j}^s$, it suffices to show that  the functions

$$u_{2j-1}-1 +h_j^*(u_1,u_3,\dots, u_{2N-1}) \; 
\Bigl(\frac{1}{4}-\sum_{j=1}^N|u_{2j-1} -1|^2\Bigr)
$$ 
have  a common zero in $D^N$.
  To do so,   let $w_j:=u_{2j-1}-1$ and
put $$h_j^{**}(w_1,\dots, w_N):=h_j^*(w_1+1,\dots, w_N+1), |w_j|\leq 1/2.$$
Since  then unit ball $\mathbf B_N=\{(z_1,\dots,z_N)\in \C^N: \sum_{j=1}^N |z_j|^2\leq 1\}$
has the property that 
$\mathbf B_N\ss \ov\D^N$, it follows that
the functions $h_j^{**}$ are continuous and bounded on $(1/2)  \mathbf B_N$.

By Brouwer's  fixed-point theorem (or no-retract theorem), see for example \cite[p. 127]{eng},
the identity map on $\partial (1/2)\mathbf B_N$
has no zero-free extension to $(1/2)\mathbf B_N$. Therefore, the map
$$\biggl(w_1+h_1^{**}(w_1,\dots,w_N)\, \Bigl(\frac{1}{4}-\sum_{j=1}^N |w_j|^2\Bigr),
\; \dots\;, \;w_N+h_N^{**}(w_1,\dots, w_N)\,\Bigl(\frac{1}{4}-\sum_{j=1}^N 
|w_j|^2\Bigr)\biggr)$$
admits a zero $\Xi:=(w_1^{(0)},\dots,  w_N^{(0)})$ in the open  ball $(1/2)\mathbf B_N$. 
 Thus we see that with
$$z_{2j-1}^{(0)}:= G_1(w_j^{(0)}+1),\;\;z_{2j}^{(0)}:=G_2(w_j^{(0)}+1), \;\;(j=1,\dots, N)$$
and $z_{j+2N}^{(0)}=\ov z_j^{(0)}$ for $j=1,\dots, 2N$, 
$(z_1^{(0)},\dots, z_{4N}^{(0)})\in \T^{4N}$ and
$$H(z_1^{(0)},\dots, z_{4N}^{(0)})=\bs 0_N.
$$
\end{proof}

 \begin{theorem}\label{bsraplambda}
The Bass and topological stable ranks of the Banach algebras ${\rm AP}_\Lambda$  are infinite
whenever the dimension of the vector space $[\Lambda]$ generated by $\Lambda$
over $\Q$ is infinite.
\end{theorem}

\begin{proof}
Let $A={\rm AP}_\Lambda$. We may suppose, without loss of generality, that $\Lambda$
contains infinitely many $\Q$-linearly independent {\it positive} reals; otherwise
the algebra 
$${\rm AP}_{-\Lambda}=\{\ov f: f\in {\rm AP}_\Lambda\},$$ 
which is isomorphic isometric to  ${\rm AP}_\Lambda$, has to be considered.

Since $\bsr A\leq \tsr A$ for any Banach algebra $A$, it suffices to show that $\bsr\,{\rm A}=\infty$.
Fix $N\in \N^*$ and consider the algebras $C(\T^{4N},\C)$ and   $A(\D^{4N})$.  
Let $\bs a:=(\bs f,g):=(f_1,\dots, f_N,g)$ be an invertible $(N+1)$-tuple in $A(\D^{4N})$ such that
$$
(\bs{\tilde  f}\bigl(z_1,\dots,z_{4N}), \tilde g(z_1,\dots,z_{4N})\bigr):=\bigl(\bs f (z_1^\nu, \dots, z_{4N}^\nu), g(z_1^\nu,\dots,  z_{4N}^\nu)\bigr)
$$
 is not reducible in $C(\T^{4N},\C)$ for any $\nu\in \N^*$ 
 (such a tuple exists by  Lemma \ref{fundamentalex}).
Then $\bs {\tilde a}:=(\bs {\tilde f}, \tilde g)$ is  not reducible in $C(\T^{m}, \C)$ for every $m\geq 4N$.

Let $\{\lambda_1,\dots,\lambda_{4N}\}$ be a set of positive reals in $\Lambda$
 that is independent over $\Q$.
Let 
$$\mbox{$\bs F(t):=\bs f(e^{i\lambda_1t},\dots ,e^{i\lambda_{4N}t})$ and
$G(t):=g(e^{i\lambda_1t},\dots ,e^{i\lambda_{4N}t})$.}$$
 We claim that
$$\bs A(t):= \bs a(e^{i\lambda_1t},\dots ,e^{i\lambda_{4N}t})$$
is an invertible $(N+1)$-tuple in ${\rm AP}_{\Lambda}$ that is not reducible in AP 
(and a fortiori not in ${\rm AP}_{\Lambda}$). 
 Assume for the moment that this
has been verified. Then we may conclude that  for $A={\rm AP}$ and $A={\rm AP}_\Lambda$,
$\bsr\,A\geq  N+1$.  Since $N$ was arbitrarily chosen,  we deduce that 
$\bsr {\rm AP}=\bsr\, {\rm AP}_\Lambda=\infty$.

Let us introduce the following notation: if $\bs a,\bs b\in A^m$, then $\bs a\cdot \bs b:=
\sum_{j=1}^m a_jb_j$.
To verify the claim, we note that $\bs a\cdot \bs b=1$ for 
$\bs a,\bs b\in A(\D^{4N})^{N+1}$
obviously implies by Lemma \ref{cnap+} that  $\bs A\cdot \bs B=1$ in ${\rm AP}_\Lambda$, where 
$$\bs B(t):= \bs b(e^{i\lambda_1t},\dots ,e^{i\lambda_{4N}t}).$$
In view of achieving a  contradiction, suppose that $\bs A$ is reducible in AP. Then there
exists $\bs H=(H_1,\dots, H_N)\in ({\rm AP})^N$ such that 
$$\bs F+ \bs H G\in U_N({\rm AP}).$$
Let $\bs F=(F_1,\dots, F_N)$.
Hence, by Theorem \ref{unap},
$$
\mbox{$\sum_{j=1}^N |F_j+H_j G|\geq \delta>0$ on $\R$}.
$$
For $j=1,\dots, N$, let $H_j^*(t)=\sum_{k=1}^{M_j} \sigma_{k,j} e^{i\lambda_{k,j}t}\in {\rm AP}$
 be chosen 
 close to $H_j(t)$ (uniformly in $t$)  so that
\begin{equation}\label{kronni}
\mbox{$\sum_{j=1}^N |F_j+H_j^* G|\geq \delta/2>0$ on $\R$}
\end{equation}
(note that  ${\rm AP}\ss C_b(\R,\C)$). 

By a standard result in linear algebra, there is 
a subset $S'$ of  $$S=\{\lambda_{k,j}: k=1,\dots,M_j,\;  j=1,\dots, N \}$$ such that the elements  in $\Lambda:=S' \union \{\lambda_1,\dots,\lambda_{4N}\} $
 are independent over $\Q$ and such that  
 $$[\Lambda]=\Bigl[ \{\lambda_1,\dots,\lambda_{4N}\}\union S\Bigr].$$
  Let $L$ be the cardinal of $\Lambda$; that is 
  $$\Lambda=\{\lambda_1,\dots,\lambda_{4N}, \lambda_{4N+1},\dots, \lambda_L\}.$$
Then, by Lemma \ref{trafo}, there exists $s\in \N^*$ such that for every $j\in\{1,\dots,N\}$,
 $$H^*_j=\Phi_{\frac{\Lambda}{{}^s}}(h_j^*)$$ for some  function  
$$h^*_j(z_1,\dots, z_L)\in C(\T^L, \C),$$
where the evaluation functional is given by
$$\Phi_{\frac{\Lambda}{{}^s}}(h)(t)=h\bigl(e^{i (\lambda_1/s) t}, \dots, e^{i (\lambda_L/s) t}\bigr).$$

Note that by Lemma \ref{cnap+}, $\Phi_{\frac{\Lambda}{{}^s}}$ is injective. Moreover,
$$\Phi^{-1}_{\frac{\Lambda}{{}^s}}(F_j)(z_1,\dots,z_{4N})=f_j(z_1^s,\dots,z_{4N}^s),$$
as well as
$$\Phi^{-1}_{\frac{\Lambda}{{}^s}}(G)(z_1,\dots,z_{4N})= g(z_1^s,\dots,z_{4N}^s).$$

Since by Kronecker's Theorem  \ref{kron2}
$$\{(e^{i (\lambda_1/s) t}, \dots, e^{i (\lambda_L/s) t}): t\in \R\}$$
is dense in  $\T^L$, we obtain  from 
 $$\sum_{j=1}^N\Bigl |f_j\bigl(e^{i\lambda_1 t},\dots, e^{i\lambda_{4N} t }\bigr)
+h_j^*\bigl(e^{(i\lambda_1/s) t},\dots, e^{(i\lambda_L/s) t }\bigr)\,
g\bigl(e^{i\lambda_1 t},\dots, e^{i\lambda_{4N} t }\bigr)\Bigr|\buildrel\geq_{\zit{kronni}}^{}
 \delta/2>0$$
that
$$\mbox{$\dis \sum_{j=1}^N |f_j(z_1^s,\dots,z_{4N}^s)+h^*_j(z_1,\dots,z_L)\,
 g(z_1^s,\dots,z_{4N}^s)|\geq \delta/2>0 $ on $\T^L$}.$$
This tells us that $(\bs {\tilde f}, \tilde g)$ is reducible in  $C(\T^L,\C)$; a contradiction.
\end{proof}

 Without the  assumption that $\dim [\Lambda]=\infty$, the Bass and topological stable ranks
of  ${\rm AP}_\Lambda$  may be one: just take $\Lambda=\Z$. Then ${\rm AP}_\Lambda$ 
coincides with the algebra of $2\pi$-periodic, continuous  functions on $\R$, which
is isomorphic isometric to $C(\T,\C)$, and $\bsr C(\T,\C)=\tsr C(\T,\C)= 1$
 (see  \cite{vas} and  \cite[p.~8]{mr}).

 More general, we have the following  result:

\begin{theorem}
Suppose that $\Lambda_0=\{\lambda_1,\dots,\lambda_N\}$  is a set of $\Q$-linearly independent, positive reals.
Let $$\Lambda_1:=\Bigl\{\sum_{j=1}^N s_j\lambda_j: s_j\in \N\Bigr\}$$ 
and
$$\Lambda_2=\Bigl\{\sum_{j=1}^N s_j\lambda_j: s_j\in \Z\Bigr\}.$$

Then $$A_1:={\rm AP}_{\Lambda_1}=\{f\in {\rm AP}: \sigma(f)\ss \Lambda_1\}$$
is a uniformly closed subalgebra of $\ap$ that is isomorphic isometric to $A(\D^N)$
and   $$A_2:={\rm AP}_{\Lambda_2}=\{f\in {\rm AP}: \sigma(f)\ss \Lambda_2\}$$
is a uniformly closed subalgebra of AP that is isomorphic isometric to $C(\T^N,\C)$.
In particular, 
$$\bsr A_1=\bsr A(\D^N)=\left\lfloor \frac{N}{2}\right\rfloor+1,\sp
\tsr A_1=\tsr A(\D^N)= N+1,$$
$$\bsr A_2=\bsr C(\T^N,\C)= \left\lfloor \frac{N}{2}\right\rfloor+1\;
\text{and}\;
\tsr A_2=\tsr C(\T^N,\C)= \left\lfloor \frac{N}{2}\right\rfloor+1.$$
\end{theorem}

\begin{proof}
Let $\tilde A_1=A(\D^N)$ and $\tilde A_2=C(\T^N,\C)$. In view of Lemma \ref{cnap+},
it suffices to show that the evaluation map 
$$\Phi_{\Lambda_0}: \begin{cases} \tilde A_j &\to A_j\\
                                        f&\mapsto \Phi_{\Lambda_0}(f),
                                         \end{cases}
                                         $$
where $\Phi_{\Lambda_0} (f)(t) := f(e^{i\lambda_1t}, \dots,    e^{i\lambda_Nt})$
is a surjection. 

Let $F\in A_1$. By Theorem \ref{mainap}, for $n\in \N^*$, there is a trigonometric polynomial 
$Q_n$ whose Bohr spectrum $\sigma(Q_n)$ is contained in $\Lambda_1$ such that
$||Q_n-F||_\infty<1/n$. Now
$$Q_n(t)=\sum_{j=1}^M a_j  e^{i \sum_{k=1}^N n_{k,j}\lambda_kt}
=\sum_{j=1}^M a_j \prod_{k=1}^N e^{i n_{k,j}\lambda_k t},$$
 where $n_{k,j}\in\N$ and $a_j\in\C$. The polynomial
 $$q_n(z_1,\dots,z_N):=\sum_{j=1}^M a_j \prod_{k=1}^N z_k^{n_{k,j}}$$
 now has the property that
 $$\Phi_{\Lambda_0}(q_n)(t)=Q_n(t).$$
 Since $\Phi_{\Lambda_0}$ is an isometry (Lemma \ref{cnap+}), we finally obtain
 that $\Phi_{\Lambda_0}(f)=F$, where $f$ is the limit point of the Cauchy sequence $(q_n)$
 in $A(\D^N)$.

 If $F\in A_2$, then we use that $e^{i m\lambda_k t} =z_k^m$ whenever $m\geq 0$
 and $e^{i m\lambda_k t} =\ov{z_k}^{|m|}$ whenever $m<0$ and proceed in a similar way
 as above.  Recall that $C(\T^N,\C)$ is, by Weierstrass' theorem,  the
 uniform closure of the polynomials in $z_j$ and $\ov z_j$.  
 
The remaining assertions follow from the corresponding results in $A(\D^N)$ and
$C(\T^N,\C)$,  since the Bass and topological stable ranks are invariant under 
isomorphic isometries.
 Recall that
 $ \bsr A(\D^N)=\left\lfloor \frac{N}{2}\right\rfloor+1$  by \cite[Corollay 3.13]{cosu},
 $\tsr A(\D^N)=N+1$ by  \cite[Theorem 3.1]{cs} and
 $\bsr C(\T^N,\C)=\tsr C(\T^N,\C)=\left\lfloor \frac{N}{2}\right\rfloor+1$ 
  by  \cite{vas} (see also \cite[p. 156-157]{moru}). Note that  the covering dimension of
   $\T^N$ is $N$. 
 \end{proof}

We would like to present  the following problem:
\begin{question}
Give  a characterization of those sub-semigroups $\Lambda$ of $(\R,+)$ for which the 
Bass stable rank of ${\rm AP}_{\Lambda}$ is finite whenever  $\dim [\Lambda]<\infty$.
What about the case where $\Lambda=\Q$?
\end{question}

 \section{The analytic trace $\ap$ of  the algebra {\rm AP}}
 
 Let $\R^+:=\{x\in \R: x\geq 0\}$ and  $\C^+:=\{z\in \C: {\rm Im} z>0\}$  the upper half-plane.
   \begin{definition}
The analytic trace $\ap$ of AP is defined as the uniform   closure in $C_b(\R,\C)$ of the set 
of all functions of the form
$$Q(t)=\sum_{j=1}^n a_j e^{i\lambda_jt},$$
where $a_j\in\C, \lambda_j\in \R^+$, and $n\in \N^*$.  \footnote{ Note that this definition coincides
with that given in the introduction in view of Theorem \ref{mainap}(2).}

Moreover, 
let $\ap_{{\rm hol}}$ denote the uniform closure in $C_b(\C^+,\C)$ of the set of all
 functions of the form 
 $$q(z)=\sum_{j=1}^n a_j e^{i\lambda_j z},$$  where
 $a_j\in\C, \lambda_j\in \R^+$, and $n\in \N^*$.
 \end{definition}

Note that the only difference in the definitions of the classes {\rm AP} and $\ap$ is that we allow here
only non-negative  exponents $\lambda_j$.
The following relations between $\ap$ and $\ap_{{\rm hol}}$ are easy to check.

\begin{theorem}\label{ap+}\hfill  
 \begin{enumerate}
\item [(1)] $ \ap_{{\rm hol}}$ is a closed subalgebra of $H^\infty(\C^+)$. 

\item [(2)] Every function $f\in  \ap_{{\rm hol}}$  has a continuous extension, $f^*$, to the boundary $\R$ of $\C^+$.
\item [(3)] $f^*\in \ap$ and $||f||_{\C^+}:=\sup_{z\in\C^+}|f(z)|= ||f^*||_\infty$.
\item [(4)]  $\ap$ is isomorphic isometric to $\ap_{{\rm hol}}$.
\item [(5)] If $g\in \ap$, then its Poisson-integral
$$[g](z):=\int_\R  P_y(x-t) g(t) \;dt,\;\; z=x+iy\in \C^+$$
belongs to $\ap_{{\rm hol}}$ and $[g]^*=g$.
\item[(6)] The Poisson operator ${\rm AP}\to C(\C^+,\C), f\mapsto [f]$ is multiplicative on $\ap$.
\end{enumerate}
\end{theorem}

The following result, that appears with an entirely different proof in \cite{bks}
does not seem to be widely known.
 
 \begin{theorem}\label{ap++}
 If $f\in H^\infty(\C^+)$ has a continuous extension $f^*$ to $\R$ such that $f^*\in {\rm AP}$,
 then $f^*\in \ap$ and $f\in \ap_{{\rm hol}}$.
\end{theorem}
 \begin{proof}
{\bf Step 1} We first show that for every $\lambda<0$ the Fourier-Bohr coefficients   $$\widehat{f^*}(\lambda)=\lim_{T\to\infty} \frac{1}{2T} \int_{-T}^T f^*(t) e^{-i\lambda t}\; dt
$$
of $f^*$ are zero.\\

Because  $\lambda<0$, the function $F(z):=f(z) e^{-i\lambda z}$ is  bounded  on $\C^+$.
Hence  $F$ belongs to $\H(\C^+)$ and has a continuous extension to $\R$.
If $M$ is an upper bound for $|f|$,  then
$$|F|\leq M e^{\lambda {\rm Im}\, z}\leq M.$$
 For $T>0$, let $\Gamma_T$ be the boundary  of the half disk
  $$\{w\in \C: |w|\leq T, \;{\rm Im}\, w\geq 0\}.$$
   By Cauchy's integral theorem, 
$$\int_{\Gamma_T} F(\xi) \; d\xi =0.$$ 
A  splitting  of the curve $\Gamma_T$  into the upper half circle  $C_T$ and the 
  segment $[-T,T]$
   yields
   \begin{equation}\label{split}
0=\int_{C_T} F(\xi)\; d\xi + \int_{-T}^T F(t) dt.
\end{equation}
   We claim that
   $$ I_T:= \frac{1}{2T}\int_{C_T} F(\xi)\; d\xi\to 0 \;\;\text{as  $T\to\infty$}.$$
   In fact, due to the symmetry of the sine-function, and  the facts that  for $0\leq \theta\leq \pi/2$ and $\xi=Te^{i\theta}$
   $$|F(\xi)|\leq  M e^{ T\lambda \sin \theta}\leq M e^{-T |\lambda | \frac{2}{\pi}\theta},$$
   we obtain
   \begin{eqnarray*}
|I_T| & \leq &  \frac{2M}{2T}\int_0^{\pi/2} e^{-T |\lambda | \frac{2}{\pi}\theta}\; T d\theta\\
&=&\frac{M}{T }\; \frac{1- e^{- T|\lambda|}}{|\lambda|\, \frac{2}{\pi}}\to 0\; \text{as $T\to\infty$}.
\end{eqnarray*}
Hence, by \zit{split},
$$\lim_{ T\to \infty}\frac{1}{2T} \int_{-T}^T F(t) dt=0.$$\medskip

{\bf Step 2} By Step 1, the Bohr spectrum 
$$\sigma(f^*)=\{\lambda\in \R: \widehat{f^*}(\lambda)\not=0\}$$
of $f^*\in {\rm AP}$ belongs to $\R^+$. Hence, by 
Theorem \ref{mainap}, we obtain a sequence  of trigonometric polynomials
$$q_n(t)=\sum_{j=1}^{N(n)} a_{j,n} e^{i\lambda _{j,n} t}$$
with $\lambda_{j,n}\geq 0$ such that $||q_n-f^*||\to 0$.
Thus $f^*\in \ap$.
\end{proof}

\begin{corollary}
$\ap$ is the set of functions in AP that admit a {\rm bounded} holomorphic extension to $\C^+$.
\end{corollary}
\begin{proof}
By Theorem \ref{ap+}, every $\ap$-function admits a bounded holomorphic extension to $\C^+$.
If, on the other hand, $F$ is a bounded holomorphic extension to $\C^+$  of $f\in {\rm AP}$,
then Theorem \ref{ap++} implies that $f=F^*\in \ap$.
\end{proof}

As a corollary to the main theorem \ref{bsraplambda} we have
\begin{theorem}
$\bsr \ap=\tsr \ap=\bsr \ap_{{\rm hol}}=\tsr \ap_{{\rm hol}}=\infty$.
\end{theorem}

\section{General subalgebras of {\rm AP}}

The most general result we obtain is the following:

\begin{theorem}\label{general}
Let $\Lambda=\{\lambda_1,\lambda_2,\cdots\}$  be a countably infinite subset of reals
that is linearly independent over $\Q$.  Associate with $\Lambda$  the following functions:
$$F_j(t):=e^{i\lambda_{2j-1}t}+ e^{i\lambda_{2j}t}-1,\sp j=1,2,\dots.$$

Then $F_j\in {\rm AP}$ and the Bass stable rank  of any (complex) subalgebra $A$ of ${\rm AP}$
 containing the functions $1$ and $F_j$~ $(j\in \N^*)$, is  infinite.  Moreover, the tuple  
 $(F_1,\dots, F_N)$ cannot be approximated  in the supremum norm by tuples invertible
 in $A$.
 \footnote{The usual results  on the topological stable rank do not apply, since these
 algebras $A$ are not necessarily complete or $Q$-algebras (i.e. $A^{-1}$ open). }
\end{theorem}

\begin{proof}
(i) By Definition of ${\rm AP}$, it is obvious that   $F_j\in AP$.  Note also that 
 $1$ is the identity element in $A$. 
Let $$G= \frac{1}{4}-\sum_{j=1}^N F_jF_{N+j}.$$
Then $G\in A$. Since $\sum_{j=1}^N F_{N+j} F_j +G = 1/4\in A^{-1}$, we see that $(F_1,\dots,F_N,G)$
is an invertible ${(N+1)}$-tuple in $A$.  By  (the proof of) Theorem \ref{bsraplambda},
$(F_1,\dots,F_N,G)$ is not reducible in ${\rm AP}$.  Consequently,
since $A\ss AP$, $(F_1,\dots,F_N,G)$ is not reducible in $A$ either.
Thus $\bsr A\geq N+1$. Since $N$ was arbitrarily chosen, $\bsr A=\infty$.\\

(ii) To prove the second assertion,  suppose that  $(F_1,\dots,F_N)$ does admit 
approximations by invertible $N$-tuples in $A$; say $||F_j-H_j||_\infty< 1/(24N)$ for some
$(H_1,\dots, H_N)\in U_N(A)$.  Since $U_n(A)\ss U_n({\rm AP})$ for every $n$,
 we conclude from $||F_j||_\infty\leq 3$ and 
$$\sum_{j=1}^N F_j F_{j+N}+ G=\frac{1}{4}$$
that
$$F:=\sum_{j=1}^N H_j F_{j+N} +G\in U_1({\rm AP}),$$
because 
$$|4F-1|= 4\Bigl|\sum_{j=1}^N (H_j-F_j) F_{j+N}\Bigr|\leq 4 N \frac{1}{24 N} 3= \frac{1}{2}.$$
Let $G_j\in {\rm AP}$ be chosen so that $1=\sum_{j=1}^N  H_jG_j$. Then 
$F$ may be rewritten as  
\begin{eqnarray*} 
F &= &\sum_{j=1}^N H_j F_{j+N} + \sum_{j=1}^N  H_jG_j G\\
&=& \sum_{j=1}^N H_j(F_{j+N}+ G_jG).
\end{eqnarray*}
Thus $(F_{N+1},\dots, F_{N+N}, G)$ is reducible  in ${\rm AP}$; a contradiction to (i)
(note that $(F_1,\dots, F_N)$ plays the same role as $(F_{N+1}, \dots, F_{N+N})$
 due to the similarity in the definitions of these functions).
\end{proof}

To sum up, under the assumptions of Theorem \ref{general}, we were
 able to determine the Bass stable ranks of all  standard
 subalgebras of ${\rm AP}$ and $\ap$  without using corona-type theorems
characterizing the invertible tuples in advance.

As  a final  important  example to which our theory applies,  we mention the Wiener-type algebra
$$APW^+:=\Bigl\{F(t):=\sum_{j=1}^\infty a_j e^{i\lambda_j t}: ||F||:=\sum_{j=1}^\infty|a_j|<\infty,\;
\lambda_j\geq 0\Bigr\};$$
see \cite{misa} for many other examples.

\bibliographystyle{amsplain}

\end{document}